\documentclass[11pt]{article}
\usepackage[utf8]{inputenc}
\usepackage{color}
\usepackage{subfigure}
\usepackage{url}
\usepackage{graphicx}
\usepackage{amsmath}
\usepackage{amsthm}
\usepackage{bm}
\usepackage[plainpages=false]{hyperref}
\usepackage{enumerate}
\usepackage{datetime}

\usepackage{titlesec}
\titleformat{\section}
	{\normalfont\Large\bfseries\filcenter}{\thesection.}{1 ex}{}
\titleformat{\subsection}
	{\normalfont\normalsize\bfseries}{\thesubsection.}{1 ex}{}
\titleformat{\subsubsection}[runin]
	{\normalfont\normalsize\bfseries\filcenter}{\thesubsubsection.}{1 ex}{}

\usepackage[charter]{mathdesign}
\usepackage[mathcal]{eucal}

\usepackage[margin=1.2 in]{geometry}

\usepackage[square,numbers,sort&compress]{natbib}
\newcommand*{\doi}[1]{doi: \href{https://dx.doi.org/#1}{\urlstyle{rm}\nolinkurl{#1}}}
\newcommand*{\arxiv}[1]{arXiv:  \href{https://arxiv.org/abs/#1}{\urlstyle{rm}\nolinkurl{#1}}}

\usepackage{thmtools,thm-restate}
\declaretheorem[within=section]{theorem}
\declaretheorem[sibling=theorem]{lemma}
\declaretheorem[sibling=theorem,name=Proposition]{prop}
\declaretheorem[sibling=theorem]{corollary}

\declaretheorem[style=remark,sibling=theorem,qed={$\diamondsuit$}]{remark}
\declaretheorem[sibling=theorem,style=definition,qed={$\blacksquare$}]{definition}
\declaretheorem[sibling=theorem]{question}
\newcommand\QQ{\mathbb{Q}}
\newcommand\RR{\mathbb{R}}

\newcommand{\defn}[1]{\emph{\color{blue} #1}} % Defined words in blue

\newcommand\p{{\bf p}}

\newcommand\rr{{\bf r}}

\newcommand\bl{{\bf l}}

\newcommand{\diff}{\mathrm{d}}
\begin{document}
\title{Rigidity for sticky disks}
\author{Robert Connelly\thanks{Department of Mathematics, Cornell University. \texttt{rc46@cornell.edu}. Partially supported by 
NSF grant DMS-1564493.}
\and Steven J. Gortler\thanks{School of Engineering and Applied Sciences, Harvard University. \texttt{sjg@cs.harvard.edu}. 
Partially supported by NSF grant DMS-1564473.}
\and Louis Theran\thanks{School of Mathematics and Statistics, University of St Andrews. \texttt{lst6@st-and.ac.uk}}}
\date{}
%\date{\currenttime\qquad \today}

\maketitle 
\begin{abstract}
We study the combinatorial and rigidity properties
of disk packings with generic radii. We show that
a packing of $n$ disks in the plane with generic 
radii cannot have more than $2n - 3$ pairs of 
disks in contact.  

The allowed motions of a 
packing preserve the disjointness of the 
disk interiors and tangency between pairs already in 
contact (modeling a collection of sticky disks).
We show that if a packing has generic radii, 
then the allowed motions are all rigid body 
motions if and only if the packing has exactly
$2n - 3$ contacts.
Our approach is to study the space of packings
with a fixed contact graph.  The main technical step
is to show that this space is a smooth manifold, 
which is done via a connection to the 
Cauchy-Alexandrov stress lemma.  

Our 
methods also apply to jamming problems, 
in which contacts are allowed to break during a 
motion.  We give a simple proof of a finite 
variant of a recent result of Connelly, et al. \cite{iso}
on the number of contacts in a jammed packing 
of disks with generic radii.

\end{abstract}

\section{Introduction}\label{sec:introduction}

In this paper we study existence and rigidity properties of packings of sticky disks
with fixed generic radii.

A (planar) \defn{packing} $P$ of $n\ge2 $ disks is a placement of the 
disks, with centers $\p = (\p_1, \ldots, \p_n)$
and fixed radii $\rr = (r_1, \ldots, r_n)$, 
in the Euclidean plane so that their interiors 
are disjoint.  The \defn{contact graph} of a packing is the 
graph that has one vertex for each disk and an edge 
between pairs of disks that are mutually tangent.
Figure~\ref{fig:prism} shows an example of a packing.

\paragraph{Sticky disk and framework rigidity}
A motion  of a packing, called
a \defn{flex}, is one  that preserves
the radii, 
the disjointness of the disk interiors,
as well as tangency between pairs of disks 
with a corresponding edge in 
the contact graph.  (The last condition makes the 
disks ``sticky''.)
A packing is \defn{rigid} when all its flexes
arise from rigid body motions; otherwise, it is \defn{flexible}.

\begin{figure}[htbp]
    \centering
    \includegraphics[width=0.4\textwidth]
    {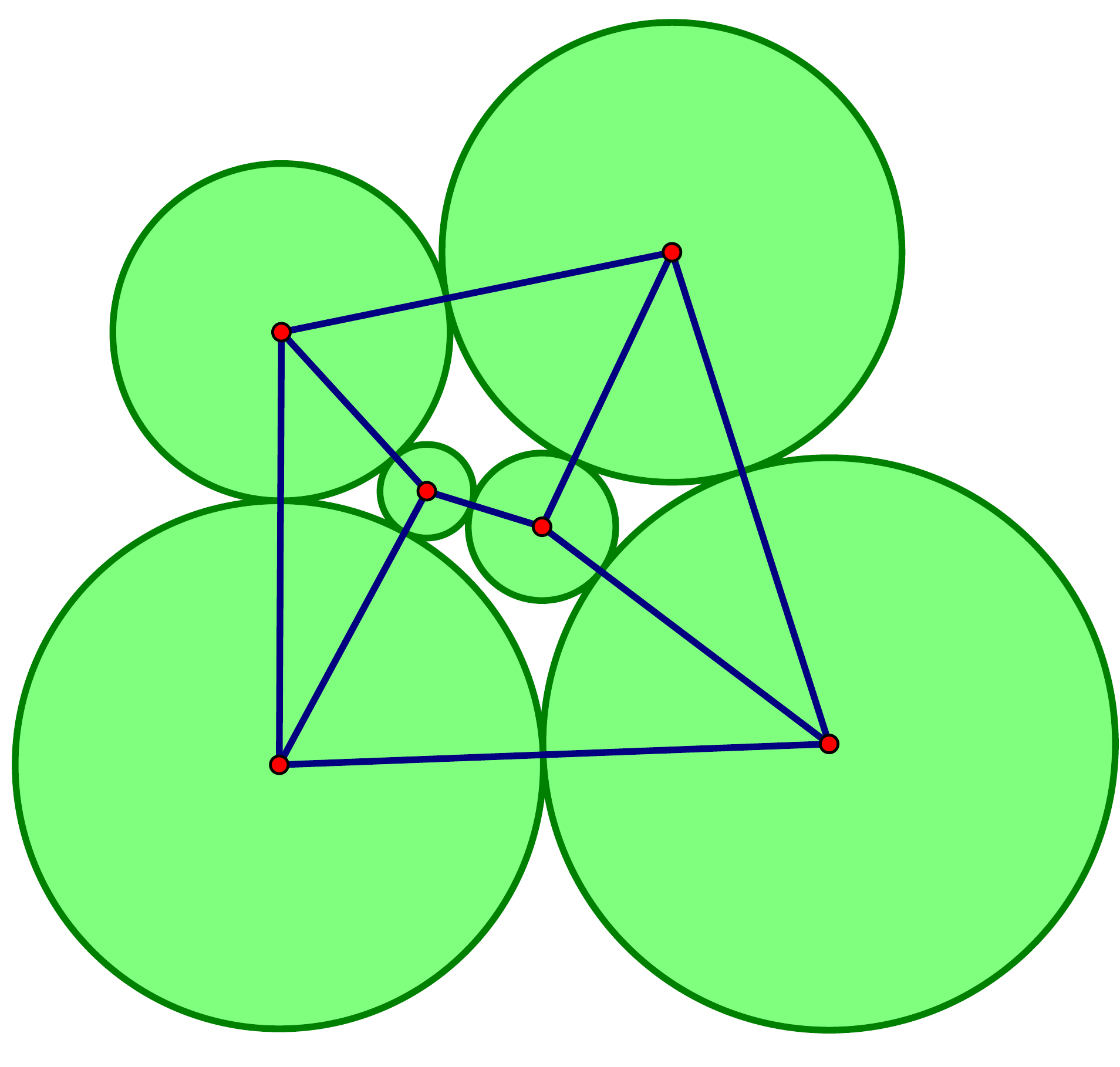}
    \caption{A rigid sticky disk packing and its underlying bar framework.  Disks 
    are the green circles with red center points / joints, the bars are blue segments.  }
    \label{fig:prism}
\end{figure}

Since any packing has a neighborhood on which 
the contact graph remains fixed along any flex (one must move at least some distance before a new contact can appear),
the constraints on the packing are locally equivalent 
to preserving the pairwise distances between the 
circle centers. 
Forgetting that the disks have radii and must remain disjoint and 
keeping only the distance contraints between the centers, we get exactly 
``framework rigidity'', where we have a \defn{configuration} 
$\p = (\p_1, \ldots, \p_n)$ 
of $n$ points in a $d$-dimensional Euclidean space 
and a graph $G$; the pair $(G,\p)$ is called a 
\defn{(bar-and-joint) framework}.
The allowed flexes of the points are those that preserve
the distances between the pairs indexed by the edges of $G$.
As with packings, a framework is \defn{rigid} when all its flexes
arise from rigid body motions and otherwise \defn{flexible}.  Given 
a packing $(\p,\rr)$ with contact graph $G$, we call the framework
$(G,\p)$ its \defn{underlying framework}.

\paragraph{Motivations}
Rigidity of sticky disk packings, and the relationship to 
frameworks, has several related, but formally different 
motivations.  
The first comes from the study of 
colloidal matter \cite{mahanoran-science}, which is
made of micrometer-sized particles that interact 
with ``short range potentials'' \cite{PhysRevLett.103.118303,Meng560}
that, in a limit, behave like sticky disks \cite{miranda-energy-landscape} (spheres 
in $3$d).

Secondly, there is a connection 
to jammed packings.  Here one has a 
``container'', which can shrink uniformly and push the disks together.
In this setting, we do not 
require any contact graph to be preserved. 
A fundamental problem is to understand the geometry and combinatorics 
of maximally dense packings (where the container can
shrink no more---full definitions 
will be given in Section \ref{sec:jam}).
In a series of papers, Connelly and co-workers \cite{C08,con88,DONEV2004139}
relate configurations that are 
locally maximally dense to 
the rigidity of a related tensegrity (see, e.g., \cite{W16}) over the
contact graph. 
Notably, the recent work in \cite{iso} proves
results about the number of contacts appearing in 
such locally maximally dense packings under appropriate genericity assumptions.

Another motivation comes from geometric constraint solving \cite{meera-EASAL},
where combinatorial methods are also applied to structures made of disks and 
spheres.

\paragraph{Laman's Theorem}
Given the relationship between disk packings  
and associated frameworks, it is very tempting to go further
and apply the methods of \textit{combinatorial} rigidity (see, e.g., 
\cite{zbMATH06607902})
theory to infer geometric or physical properties from the contact 
graph alone.
Several 
recent works in the soft matter literature 
\cite{PhysRevLett.114.135501,LESTER2018225,PhysRevLett.116.028301} use such an approach.

The combinatorial approach is attractive 
because we have a 
very good understanding of framework rigidity in  dimension $2$, provided 
that $\p$ is not very degenerate.
\begin{definition}
A vector in $\RR^N$ is called \defn{generic} if its coordinates
are algebraically independent over $\QQ$.
A point configuration $\p$ of $n$ points in $\RR^d$
is called \defn{generic} 
if the coordinates of its points (a vector in $\RR^{dn}$)
are algebraically independent over $\QQ$.
\end{definition}
Almost all configurations are generic, so generic 
configurations capture the general case.  
Moreover, all the results in the present paper 
remain true if we simply  avoid the zero set of a 
specific but unspecified set of
polynomials with rational coefficients (i.e., the results
hold on a Zariski open subset, defined over $\QQ$,
of the appropriate configuration space).

The following combinatorial notion captures the 
$2$-dimensional case of a counting heuristic 
due to Maxwell \cite{M64}.
\begin{definition}
Let $G = (V,E)$ be a graph with $n$ vertices and $m$ edges.  A 
graph $G$ is \defn{Laman-sparse} if, for every subgraph 
on $n'$ vertices and $m'$ edges, $m' \le 2n' - 3$.  If, in 
addition, $m=2n-3$, $G$ is called a \defn{Laman graph}.
\end{definition}
The following theorem, which combines results of Asimow and Roth \cite{AR78}
and Laman \cite{L70} characterizes rigidity and flexibility of generic 
frameworks in the plane.
\begin{theorem}\label{thm: AR laman}
Let $G$ be a graph with $n$ vertices and $\p$ a generic 
configuration of $n$ points in dimension $2$. Then the framework
$(G,\p)$ is rigid if $G$ contains a Laman graph as a spanning 
subgraph and otherwise flexible.
\end{theorem}
Theorem \ref{thm: AR laman} makes generic rigidity and flexibility 
a combinatorial property that can be analyzed very efficiently
using graph-theoretic algorithms \cite{JH97,BJ03}, even on
large inputs.  Of course, one needs to make the 
modeling assumption that the process generating $\p$ 
is generic.  If not, then the use of combinatorial methods
is not formally justified.

In fact, the genericity\footnote{Or at least restricting $\p$ to a 
Zariski open set.}
hypothesis is essential: one may find (necessarily 
non-generic) $\p$ for which a framework $(G,\p)$ is flexible 
and $G$ is a Laman graph (see, e.g., \cite{WW83}); conversely, 
there are non-generic frameworks that are rigid but have too 
few edges to be rigid with generic $\p$ (see, e.g., \cite{CT95})%
\footnote{Many examples of both types are classically
known in the engineering literature.}.

\paragraph{Packing Laman question}
The starting point for this paper is the observation
that configurations arising from packings with 
disks in contact are not generic. Thus the theory of
generic bar frameworks does not apply directly to 
packings.

The most 
general situation for packings is the 
(extremely) ``polydisperse'' case in which 
the radii $r_i$ are algebraically independent
over $\QQ$.  Even with generic radii, the configuration of 
disk centers could be very 
degenerate relative to picking $\p$ freely.  For an 
edge $ij$ in $G$, the contact constraint is 
\begin{equation}\label{eq: contact constraint}
    \|\p_i - \p_j\| = r_i + r_j
\end{equation}
which means that there are only $n$ 
degrees of freedom (the radii) to pick the edge 
lengths, instead of the $2n-3$ available
to a general $(G,\p)$ with $G$ a Laman graph.
Thus  the underlying $\p$ of a disk packing with a Laman contact
graph is
certainly not generic, and
so Theorem \ref{thm: AR laman}
does not apply.  We need another formal justification 
for analyzing packings combinatorially.

\paragraph{Packing non-existence question}
Beyond rigidity, there is the general question of what
graphs can appear as the contact graph of a 
packing with generic radii.  Once we 
fix $\rr$, there are $2n-3$ non-trivial 
degrees of freedom in packing $\p$.
If a graph $G$ has $n$ vertices, each of 
its $m$ edges contributes a constraint of 
the type \eqref{eq: contact constraint}.
Thus, when $m > 2n - 3$, we expect, 
heuristically that either no $\p$ 
exists or $\rr$ satisfies some additional
polynomial relation.

\subsection{Main result}
Our main result is a variant of Theorem \ref{thm: AR laman}
for packings that answers both the 
rigidity and non-existence questions under the assumption
of generic radii.
To state our theorem, we need one more rigidity concept, which is a 
linearization of rigidity.
\begin{definition}\label{def: inf flex}
An \defn{infinitesimal flex} $\p'$ of a $d$-dimensional 
bar-and-joint framework $(G,\p)$
is an assignment of a vector $\p'_i\in \RR^d$ to each vertex of $G$ so that
for all edges  $ij$ of $G$,
\[
    (\p_j - \p_i)\cdot (\p'_j - \p'_i) = 0.
\]
There is always a $\binom{d+1}{2}$-dimensional space 
of \defn{trivial} infinitesimal flexes arising from 
Euclidean isometries.  A framework is \defn{infinitesimally
rigid} if all its infinitesimal flexes are trivial.  Infinitesimal
rigidity implies rigidity.

An infinitesimal flex for a disk packing is simply an 
infinitesimal flex of its  underlying
bar-and-joint framework.  A packing is infinitesimally rigid if its 
underlying bar-and-joint framework is.
\end{definition}
Since rigidity does not imply infinitesimal rigidity, infinitesimal 
rigidity is a more stringent condition to place on a framework than
rigidity.  Generically, however, the two concepts coincide \cite{AR78}.

\begin{figure}[htbp]
    \centering
    \includegraphics[width=0.4\textwidth]    {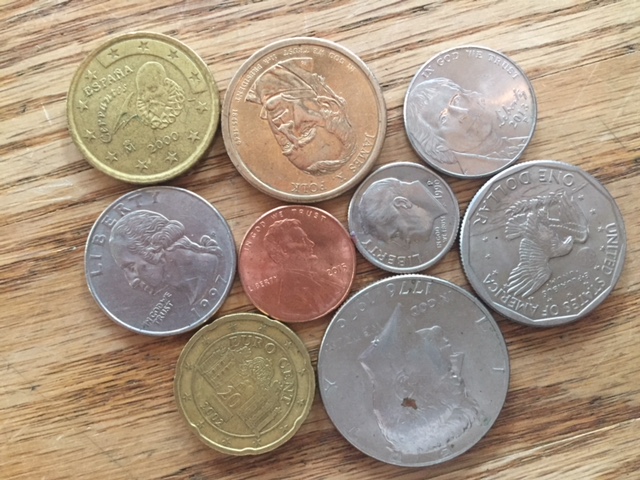}
    \caption{Nine disks behaving generically. There are fifteen contacts. One cannot create any more contacts. 
    One cannot deform the configuration without breaking a contact. }
    \label{fig:coins}
\end{figure}

Here is our main result:
\begin{theorem}\label{thm:main}
Let $P$ be a packing of $n$ disks in in the Euclidean plane
$\RR^2$ with generic radii (or even just
generic radii ratios).
Then the contact graph $G$ of $P$ has at most $2n-3$ edges and is 
Laman sparse and planar. 
Moreover, if $P$ has $2n-3$ 
contacts, 
it is rigid and infinitesimally rigid.
If the number of contacts is 
fewer, then $P$ is flexible and infinitesimally 
flexible.
\end{theorem}
Figure \ref{fig:coins} shows a ``real world'' example of a 
packing with radii that exhibit generic behavior.

We note that the upper bound on the number of contacts does not 
rely on any rigidity properties of $P$, and indeed, a similar 
upper bound will appear in Section \ref{sec:jam}, even though
the rigidity statement is different.

The rigidity characterization has an algorithmic consequence.
Since it only requires checking the total number of contacts,
generic rigidity of a graph known to be the contact graph of a 
disk packing with generic radii can be checked in linear time.
In contrast, for general graphs, the Laman sparsity condition 
must be verified for all subgraphs.  The best known algorithms
\cite{BJ03,JH97,gabow}
for this task have super-linear running times.

As with Theorem \ref{thm: AR laman}, genericity is essential to 
Theorem \ref{thm:main}.  
If all the radii are the same, 
the triangular lattice packing gives a packing with 
more than $2n-3$ contacts.  More generally, the Köbe-Andreev-Thurston
Theorem  \cite[Theorem 13.6.2]{gt3m} says that
any planar graph can appear as the contact graph of some packing
(which  necessarily 
has non generic radii when $m > 2n -3$).

\begin{figure}[htbp]
    \centering
    \subfigure[]{\includegraphics[width=0.35\textwidth]{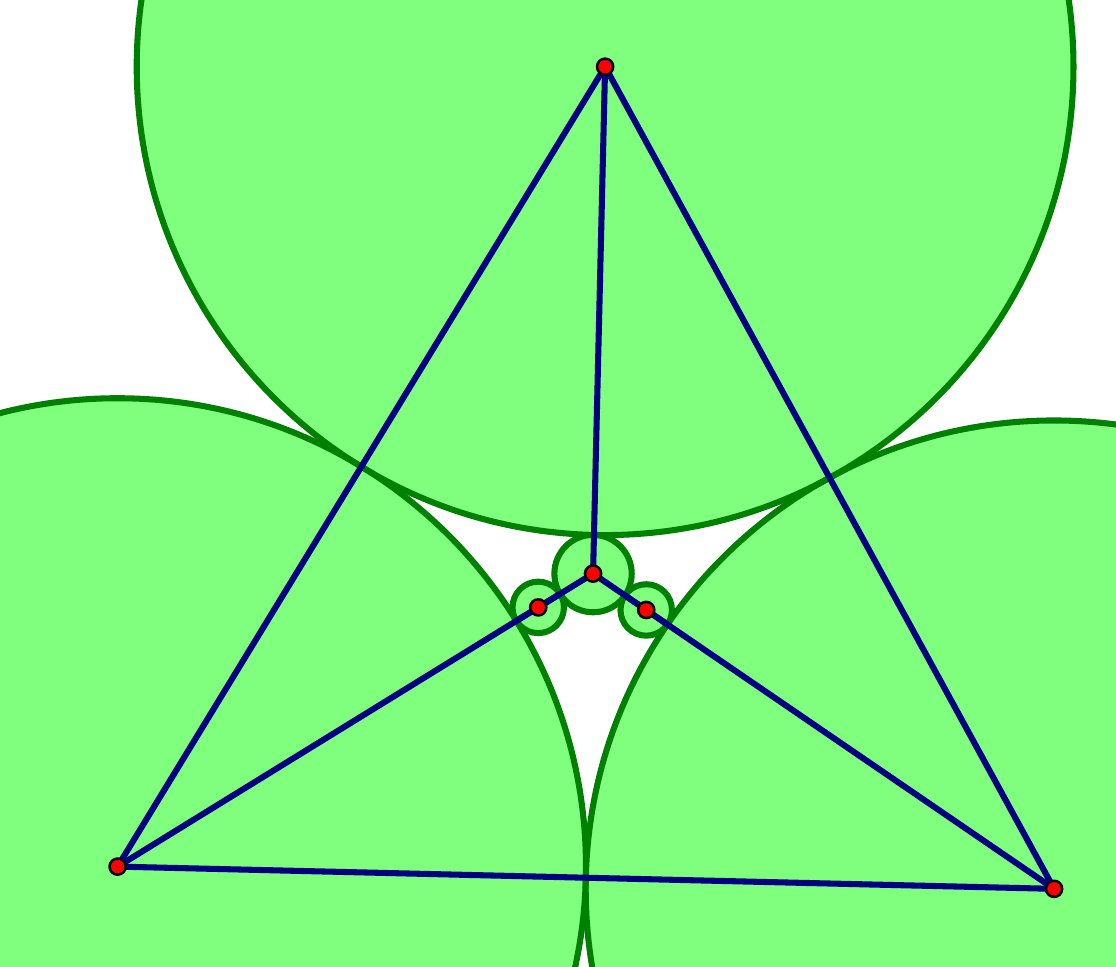}}
    \subfigure[]{\includegraphics[width=0.3\textwidth]{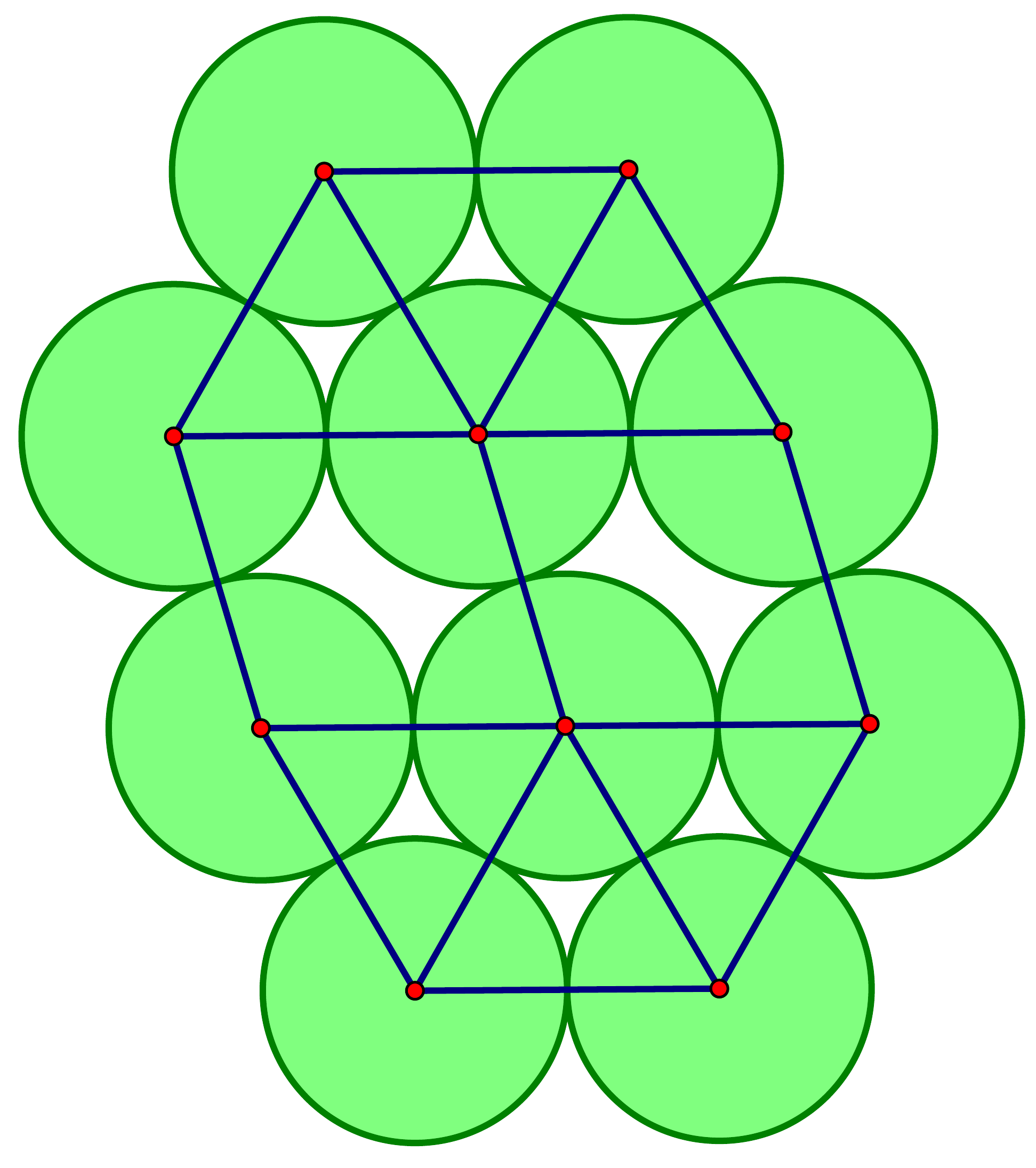}}
    \caption{
    (a) This packing of sticky disks with non-generic radii has 
    $< 2n-3$ contacts but is rigid (See \cite[Figure 8b]{CW-prestress}).
    (b) This packing of sticky disks with non-generic radii has 
    $2n-3$ contacts but is flexible.}
    \label{fig:flex}
\end{figure}

One can also construct non-generic examples of packings with fewer than 
$2n-3$ contacts that are rigid (see Figure~\ref{fig:flex}a) and at least $2n-3$ contacts that are 
flexible (see Figure~\ref{fig:flex}b).

In Section~\ref{sec:jam} we also provide a bonus result
characterizing the number of contacts that appear 
in a maximally jammed packing where the boundary is formed
by three touching large exterior disks.

\subsection{Other related work}
General questions about of whether combinatorial characterizations 
of rigid frameworks remain valid in the presence of special 
geometry have been addressed before.  Notably, our 
``packing Laman question'' is similar in flavor to the 
``Molecular conjecture'' of Tay and Whiteley \cite{TW84},
which was solved by Katoh and Tanigawa \cite{KT11}.

\section{The packing manifold}\label{sec: packing manifold}

To prove our theorem, we start by defining the
packing manifold of a contact graph.

\begin{definition}
Let $G$ be a graph with $n$ vertices and $m$ edges.
We think of a disk packing $(\p,\rr)$ as a point in $\RR^{3n}$.
Let $S_G \subset \RR^{3n}$ be the set of disk packings  
where (only) the edges 
in $G$ correspond
to the circle pairs that are in contact. 
$S_G$ is a semi-algebraic set defined over $\QQ$.
\end{definition}

Our  goal in this section is to prove that 
(when not empty)
$S_G$ is a smooth manifold
and of the expected dimension $3n - m$.

Let $P$ be a point of $S_G$.  Since exactly the pairs of disks corresponding to the
edges of $G$ are in contact, the constraints defining $S_G$ near $P$ are
$m$ equations of the type \eqref{eq: contact constraint}.  

We now compute the Jacobian matrix $M$ for these constraints at $P$. 
The matrix $M$ is $m$-by-$3n$, with one row per contact edge, and three 
columns corresponding to each vertex $i$ of $G$: two for $\p_i$ and 
one for $r_i$.

Differentiating \eqref{eq: contact constraint} with 
respect to the coordinates of 
$\p_i$, $\p_j$, $r_i$ and $r_j$ in turn, we find that
the row corresponding to an edge $ij$ of $G$
has the following pattern
\begin{equation}
\label{eq: contact jacobian}
       \bordermatrix{
        & \cdots & \p_i& \cdots & \p_j & \cdots & r_i & \cdots & r_j \cr 
        & \cdots\, 0\, \cdots & 2(\p_i - \p_j) &  \cdots\, 0\,\cdots & 
        2(\p_j - \p_i) &  \cdots\, 0\,\cdots & -2\|\p_i - \p_j\| &
         \cdots\, 0\,\cdots & -2\|\p_i - \p_j\|
    } 
\end{equation}
where the first row above labels the matrix columns.
Here, we used the fact that disks $i$ and $j$ are in contact to 
make the simplification 
\[
-2(r_i + r_j) = -2\|\p_i - \p_j\|
\]
in \eqref{eq: contact jacobian}.

To prove that $S_G$ is smooth, we only need to show that 
$M$ has rank $m$ at every 
point $P$.  We will establish this via a connection between 
row dependencies in $M$ and a specific kind of 
equilibrium stress in planar frameworks.
\begin{definition}
A \defn{edge-length equilibrium stress}
$\omega$ of a framework $(G,\p)$
is a non-zero vector in $\RR^m$ that satisfies
\begin{eqnarray}
\label{eq:coker1}
\sum_j \omega_{ij} (\p_i-\p_j)&=&0\\
\label{eq:coker2}
\sum_j \omega_{ij} \|\p_i-\p_j\|&=&0
\end{eqnarray}
for each vertex $i\in V(G)$. The sums in \eqref{eq:coker1}--\eqref{eq:coker2}
are over neighbors $j$ of $i$ in $G$.

An edge-length equilibrium stress is \defn{strict} if it has no zero coordinates.

\end{definition}
Vectors $\omega$ satisfying only \eqref{eq:coker1} are called 
\defn{equilibrium stresses}
and play a fundamental role in the theory 
of bar-and-joint frameworks.  Equation \eqref{eq:coker2} is the new
part of the definition which is relevant to packings.
\begin{remark}
W. Lam (private communication) 
has shown that edge-length equilibrium stresses
are equivalent to the holomorphic quadratic 
differentials of Köbe type from \cite{lam2}.  In particular, 
for planar frameworks without boundary (in the sense of 
\cite{lam2}), there are none.

Interestingly, the main theorem of \cite{lam2}
says that a planar embedded framework has an 
edge-length equilibrium stress if and only 
if a triangulation of its dual medial graph has 
an equilibrium stress satisfying the formally
different condition  $\sum_j \omega_{ij}\|\p_i - \p_j\|^2 = 0$
at each vertex.  This condition is connected to orthogonal 
circle patterns \cite{lam2} and discrete minimal surfaces 
\cite{lam}.
\end{remark}
We defined edge-length equilibrium stresses because they 
are the co-kernel vectors of $M$.
\begin{lemma}
\label{lem:coker}
Let $G$ be a graph, $P\in S_G$ be a packing, and $(G,\p)$ its 
underlying framework.  Then $\omega$ is an edge-length equilibrium stress
of $(G,\p)$ if and only if $\omega$ is in the co-kernel of the 
packing constraint Jacobian $M$.
\end{lemma}
\begin{proof}
Equations \eqref{eq:coker1} and \eqref{eq:coker2} are equivalent 
to $\omega^t M = 0$ written down column by column.
\end{proof}

Next we want to show that there can be no co-kernel vector for the underlying
framework of a disk packing $P$.
We will do this by showing a stronger statement, namely that there can be no edge-length 
equilibrium  stress for \emph{any} bar-and-joint framework with a planar embedding.
(When $m > 2n - 3$, there will always be equilibrium stresses of any framework
$(G,\p)$.  However, none of them will be 
an edge-length equilibrium stress, satisfying Equation (\ref{eq:coker2})
when $(G,\p)$ has a planar embedding.)

\begin{definition}
We say that a framework $(G,\p)$ is a 
\defn{planar embedded framework} if all the 
points in the configuration $\p$ are distinct 
and correspond to the vertices of a non-crossing, straight 
line drawing of $G$ in the plane.  (In particular, the existence of a
planar embedded framework $(G,\p)$ implies that $G$ is a planar graph.)
\end{definition}

\begin{definition}
Given a planar embedded framework $(G,\p)$ and 
a strict edge-length equilibrium  stress $\omega$, we can assign 
a sign in $\{+,-\}$ to each undirected edge $ij$ using 
the sign of $\omega_{ij}$. This assignment gives us a \defn{sign vector}.

Given a sign vector on a planar embedded framework
we can define the \defn{index} $I_i$ as 
the number of times the sign changes as we traverse the 
edges in order around vertex $i$.  (We use the planar embedding to 
get the cyclic ordering of edges at each vertex.)

The index $I_i$ is always even.
\end{definition}

The following is Cauchy's index lemma, which can be proven using Euler's 
formula. For a proof, see e.g.,~\cite[Lemma 5.2]{gluck}
or~\cite[Page 87]{AlexConv}.
\begin{lemma}
\label{lem:topo}
Let $(G,\p)$ be a planar embedded framework and let
$s$ be a  sign vector.
Then $\sum_i I_i \leq 4n-8$. Thus there must be at least one vertex 
with index of either $0$ or $2$.
\end{lemma}

Next, we establish the following geometric lemma.
\begin{lemma}
\label{lem:geom}
Let $(G,\p)$ be a planar embedded framework with a strict edge-length equilibrium  
stress vector $\omega$. Let $s$ be the associated sign vector and $I$ be the associated 
index vector. For each vertex $i$, the index $I_i$
is at least $4$.
\end{lemma}
\begin{figure}[htbp]
    \centering
    \includegraphics[width=0.25\textwidth]{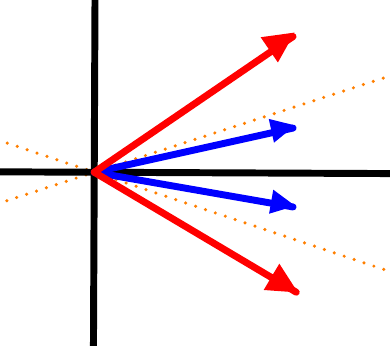}
    \caption{Illustration of the proof of Lemma \ref{lem:geom}.  This 
    vector configuration has only two sign changes (blue are negative signs and 
    red are positive) and satisfies the equilibrium condition \eqref{eq:coker1}
    with coefficients $\pm 1$. This forces the sum of the lengths of the
    red edges to be larger than that of the blue edges, so \eqref{eq:coker2} 
    is violated.}
    \label{fig: wedge}
\end{figure}
\begin{proof}
Suppose some vertex $i$ has fewer than $4$ sign changes. 
If it has $0$ sign changes, then it cannot satisfy 
Equation (\ref{eq:coker2}), as all lengths are positive. 
So now let's suppose that it has $2$ sign changes.

With only $2$ sign changes, the edges from at least one of the signs (say $-$) must
be in a wedge of angle $2\theta < \pi$.

Euclidean images of $\p$ have the same edge-length equilibrium stresses, 
so we may assume that the positive part of the  $x$-axis
is the bisector of the wedge.  
The $2$D equilibrium condition of Equation (\ref{eq:coker1}) must hold 
after projection along any direction, including onto the $x$-axis,
since \eqref{eq:coker1} is invariant under any affine transformation
(see, e.g., \cite{MR2504741}).

Let $N^+$ denote the neighbors of $i$ connected by edges with 
positive sign and $N^-$ the neighbors connected by negatively signed
edges.
Let $p^x_i$ be the $x$-coordinate of the point $\p_i$.
We then get:
\begin{eqnarray*}
\sum_{j \in N^+} \omega_{ij} (p^x_i-p^x_j)
=
\sum_{j \in N^-} -\omega_{ij} (p^x_i-p^x_j)
\end{eqnarray*}
But for $j \in N^+$ (outside the wedge), we have 
\begin{eqnarray*}
(p^x_i-p^x_j) < \cos(\theta)\; \|\p_i-\p_j\| 
\end{eqnarray*}
while for $j \in N^-$ (inside the wedge),  we have 
\begin{eqnarray*} 
(p^x_i-p^x_j) > \cos(\theta)\; \|\p_i-\p_j\|
\end{eqnarray*}
Putting these estimates together we have 
\[
    \sum_{j\in N^+} \omega_{ij} \cos(\theta)\; \|\p_i-\p_j\| 
    > \sum_{j \in N^+} \omega_{ij} (p^x_i-p^x_j)
    = \sum_{j \in N^-} -\omega_{ij} (p^x_i-p^x_j)
    > \sum_{j\in N^-} -\omega_{ij} \cos(\theta)\; \|\p_i-\p_j\| 
\]
which means that Equations \eqref{eq:coker1} and \eqref{eq:coker2} cannot 
hold simultaneously.
\end{proof}
See Figure~\ref{fig: wedge} for an illustration of this argument.
Understanding which, necessarily non-planar, frameworks 
can have edge-length equilibrium stresses would be interesting.
\begin{remark}
Lemma \ref{lem:geom} can also be reduced to the 
the Cauchy-Alexandrov stress lemma (see \cite[Lemma 5.2]{gluck}), 
which says that an equilibrium stress 
satisfying Equation (\ref{eq:coker1})  in 3D must have
at least 4 sign changes at a strictly convex vertex of a 
polytope. 

In the reduction, we place a vertex with two sign changes 
at the origin and lift its neighbors from 
$(x,y)$ to $(x,y,\sqrt{x^2+y^2})$. If the 3D equilibrium
equation holds \eqref{eq:coker1} at this vertex, 
then both of $\eqref{eq:coker1}$ and \eqref{eq:coker2}
hold in the plane.

The index and stress lemmas are the two central ingredients in a  
proof of Cauchy's theorem on the (infinitesimal) rigidity of 
convex polyhedra (see \cite{AlexConv}).
\end{remark}

Putting together the index lemma \eqref{lem:topo} and geometric lemma \eqref{lem:geom} 
we get:
\begin{lemma}
\label{lem:noco}
Let $(G,\p)$ be a planar embedded framework. Then it cannot have a non-zero edge-length 
equilibrium  stress  vector.
\end{lemma}
\begin{proof}
If $(G,\p)$ has an edge-length equilibrium  stress vector, then by removing edges 
with $0$ stress coefficients, we obtain
a subframework $(G',\p')$ with a strict edge-length equilibrium  stress  vector. 
A strict edge-length equilibrium  stress vector would form a contradiction 
between Lemmas~\ref{lem:topo} and~\ref{lem:geom}.
\end{proof}

\begin{remark}
A similar statement and proof to Lemma \ref{lem:noco}, phrased 
in terms of inversive distances \cite{MR0203568}, 
was found independently by Bowers, Bowers and Pratt \cite{bowbow-circlepolyhedra}.
\end{remark}

We are ready to prove the main result of this section.

\begin{prop}
\label{prop:man}
The set $S_G$, if not empty, is a smooth submanifold of dimension $3n-m$.
\end{prop}
\begin{proof}
Let $P$ be a point in $S_G$. It has the edges of $G$ in contact 
and no other pairs of disks in contact
in a sufficiently small neighborhood. Thus restricted to this neighborhood, $S_G$ is 
exactly defined by the
contact constraints of Equation \ref{eq: contact constraint}.
Meanwhile, the contact graph creates a planar embedded framework.
From Lemmas~\ref{lem:coker} and ~\ref{lem:noco}
there can be no co-kernel vector, and hence the Jacobian matrix 
has rank $m$. By the implicit function theorem,
$S_G$ restricted to some neighborhood of $P$ is a smooth manifold of 
dimension $3n-m$.
\end{proof}

\begin{remark}
\label{rem:kat}
The Köbe-Andreev-Thurston theorem implies that $S_G$ is non-empty provided
that $G$ is planar. (See \cite[Theorem 13.6.2]{gt3m}.)
\end{remark}

\section{Finishing the Proof}\label{sec: rad proj}

\begin{definition}\label{def: rad proj map}
Let $\pi$ be the projection $(\p,\rr)\mapsto \rr$
taking a disk packing  to its vector of radii 
in $\RR^n$.
\end{definition}

\begin{definition}\label{def: pi kernel vec}
A \defn{$\pi$-kernel}
vector of a disk packing $P$ with contact graph $G$ is  
a tangent vector to $S_G$ of the form $(\p',0)$.
These vectors form the kernel of the linearization
of the map $\pi$.
\end{definition}

\begin{lemma}
\label{lem:piker}
The $\pi$-kernel vectors of a disk packing $(\p,\rr)$ with contact
graph $G$ are
infinitesimal flexes of the underlying bar-framework $(G,\p)$.
\end{lemma}
\begin{proof}
As in the proof of Proposition~\ref{prop:man}, we have the $m$-by-$3n$ Jacobian matrix
$M$ with rank $m$. Tangent vectors to $S_G$ are thus (right) kernel vectors
$(\p',\rr')$
of $M$. From Equation \eqref{eq: contact jacobian}, the  tangent vectors
$(\p',\rr')$ to $S_G$ are exactly the vectors satisfying
\begin{equation}
\label{eq:stan}
        (\p_j - \p_i)\cdot (\p'_j - \p'_i) =
    (r_j + r_i) (r'_j + r'_i).
\end{equation}

$\pi$-kernel vectors are defined as  the tangent vectors with $\rr'=0$, 
giving us
\begin{equation*}
        (\p_j - \p_i)\cdot (\p'_j - \p'_i) = 0.
\end{equation*}
Thus $\pi$-kernel vectors
of $M$ are exactly 
the infinitesimal flexes of
the underlying framework $(G,\p)$.
\end{proof}

We now recall a standard definition and result
from differential topology.
\begin{definition}
Let $X$ and $Y$ be smooth manifolds of dimension 
$m$ and $n$, respectively.  A smooth 
map $f : X\to Y$ is a \defn{submersion} at a 
point $x\in X$ if the linearization 
$\diff f_x : T_x X\to T_{f(x)} Y$ at $x$ is surjective.  A
point $x\in X$ is called a \defn{regular point} of $f$
if $f$ is a submersion at $x$; otherwise $x$ is a 
\defn{critical point} of $f$.

A point $y\in Y$ is called a \defn{regular value}
of $f$ if $f^{-1}(y)$ consists of only 
regular points (or is empty); otherwise $y$ is a \defn{critical
value}.
\end{definition}
The following is the semi-algebraic version of 
Sard's theorem.  
\begin{theorem}\label{thm: semi sard}
Let $X$ and $Y$ be smooth semi-algebraic manifolds 
of dimensions $m$ and $n$
defined over $\QQ$ and $f : X\to Y$ a rational map.
Then the critical values of $f$ are a semi-algebraic 
subset of $Y$, defined over $\QQ$, and of dimension 
strictly less than $n$.
\end{theorem}
\begin{proof}
That the critical values are semi-algebraic and of lower 
dimension is \cite[Theorem 9.6.2]{BCR98}.  That the 
field of definition does not change follows from the 
fact that the critical points lie in a semi-algebraic subset 
defined over $\QQ$ (by the vanishing of a determinant), and
then the critical values do because quantifier elimination 
preserves field of definition \cite[Theorem 2.62]{BPR03}.
\end{proof}

\begin{lemma}
\label{lem:sard}
Let $\rr$ be a
a generic point in $\RR^{n}$ and 
let $P:=(\p,\rr)$ be a disk packing 
with $m$ contacts and contact graph $G$.
Then 
the linear space of $\pi$-kernel vectors 
is of dimension $2n-m$. 

Moreover, the set of packings with radii $\rr$ and contract graph $G$ 
form a smooth, semi-algebraic manifold of 
the same dimension $2n - m$.
\end{lemma}
\begin{proof}
Because $P$ exists, $S_G$ is non-empty, and so, by 
Proposition \ref{prop:man}, is a smooth semi-algebraic 
manifold of dimension $3n - m$.  The map $\pi : S_G\to \RR^n$
is a polynomial map,  so Theorem \ref{thm: semi sard} applies, making all 
the critical values of $\pi$ non-generic.

Since $\rr$ is generic, it must be a regular value of $\pi$.
Hence $P$ is a regular point, and the linearization 
of $\pi$ at $P$ has rank $n$.  Its kernel then 
has dimension $(3n - m) - n = 2n - m$.

Finally, the set of packings with radii $\rr$ and 
contact graph $G$ is simply $\pi^{-1}(\rr)$.  The 
preimage theorem (see, e.g., \cite[p. 21]{gp-diff}),
implies that $\pi^{-1}(\rr)$ is smooth and of 
the same dimension as the kernel of $\diff \pi_P$.
\end{proof}
The rank of $\diff \pi_P$ is never larger than 
the dimension of $S_G$, which is $3n - m$. 
If the dimension of $S_G$ is less than $n$,
$\diff \pi_P$ cannot be surjective for any $P$ 
in $S_G$, making every point in $S_G$ a critical point.
Hence, for $\rr$ to be generic, we must have $m \le 2n$.  
We will improve the preceding bound on $m$ momentarily.
\begin{remark}
The proof of Lemma \ref{lem:sard} shows that $\rr$ does 
not need to be generic for the conclusion 
to hold.  It just needs to be a regular
value of $\pi$.  By Theorem \ref{thm: semi sard},
the regular values of $\pi$ are a Zariski open subset
of $\RR^n$, defined over $\QQ$.
\end{remark}

We can now prove our main result.

\begin{proof}[Proof of Theorem~\ref{thm:main}]
Since we are in dimension $d=2$, there is a 
$3$-dimensional space of trivial infinitesimal
motions of any framework $(G,\p)$.

The packing $P := (\p,\rr)$ has $\rr$ generic
by hypothesis.  Lemma \ref{lem:sard} then 
implies that the space of $\pi$-kernel vectors
has dimension $2n - m$.  The presence of a $3$-dimensional
space of trivial infimitesimal motions then implies that
$2n - 3 \ge 3$, so $m\le 2n - 3$.
The same argument, applied to each subgraph of $G$,
shows that $G$ is Laman-sparse.

If $G$ has $m = 2n - 3$ edges, then Lemmas \ref{lem:piker}
and \ref{lem:sard} imply that $(G,\p)$ has only a $3$-dimensional
space of infinitesimal flexes and is thus infinitesimally
rigid and hence rigid.  Otherwise $m < 2n - 3$, 
and so the  $\pi^{-1}(\rr)$ has dimension at 
least $4$ by Lemma \ref{lem:sard}.  As the 
space of frameworks 
related to $(G,\p)$ by rigid body motions is only $3$-dimensional,
it follows that $(G,\p)$ is flexible.
A flexible framework is always infinitesimally flexible.

Since rigidity properties are invariant with respect to global scaling, we only need 
genericity of the radii ratios.
\end{proof}

\section{Isostatic Jamming}
\label{sec:jam}
In this section we use the technology developed above
to prove a result about ``jammed packings''.
In light of~\cite{iso}, the fact that this can be proven is not surprising. 
Our contribution is to show how the elementary methods we used to 
treat sticky disks adapt easily to jamming questions.

\paragraph{Rigidity preliminaries for jamming}
Loosely speaking we consider a set of disks in the plane 
with a fixed set of radii. We don't allow the disks to 
ever overlap. The disks are not sticky.
Now we suppose that there is some boundary shape surrounding these
disks that is uniformly shrinking. The boundary will start pressing the disks
together, and eventually the boundary can shrink no more. 
At this point in the process, 
we will say that the packing is ``locally maximally dense''. 

An equivalent way to study this problem is to assume that the
boundary shape stays fixed and that the disk radii are all scaling
up uniformly, maintaining their radii ratios. 
We will use this interpretation. 

The literature considers a number of different boundary 
shapes, including convex polyhedra (e.g., \cite{con88})
and flat tori (e.g., \cite{iso,PhysRevLett.114.135501}).
Here, we consider
a boundary  formed by three large touching
exterior disks.  This kind of boundary, which we will
call a ``tri-cusp'', has the advantage that it can be 
modeled by the same type of constraints as those on 
the interior disks.

\begin{definition}
 A \defn{packing inside of a tri-cusp}
 is a disk packing in the plane where 
the first three disks are in mutual contact, and the
remaining $n-3$ ``internal''
disks are in the interior tri-cusp shape bounded by these first $3$ disks. 
The packing will have some contact graph
$G$ that includes the triangle $\{1,2,3\}$.
See Figure~\ref{fig:nmd}.
\end{definition}

To represent the internal radii ratios
we define, for $j=4,5,\ldots, n-1$ the ratio
$\bar{r}_j:=r_j/r_n$.

\begin{definition}
We say that a disk packing in a  tri-cusp is 
\defn{locally maximally dense} if
there is no nearby tri-cusp
packing with the same $\{r_1, r_2, r_3, 
\bar{r}_4,\ldots,\bar{r}_{n-1}\}$
that has a higher area of coverage within the tri-cusp. 
The outer three disks must maintain contact.
\end{definition}

Local maximal density can be studied using a  notion 
related to infinitesimal rigidity.
\begin{definition}
Given  a disk packing  
in a tri-cusp $(\p,\rr)$, with contact graph $G$,
an \defn{infinitesimal tensegrity flex} is 
 a vector
$\p'$
that satisfies
\begin{eqnarray}
\label{eq:tflex0}
     (\p_j - \p_i)\cdot (\p'_j - \p'_i) 
    = 0, 
\end{eqnarray} 
on the three edges of the outer triangle (see Figure \ref{fig:nmd}) and satisfies
\begin{eqnarray}
\label{eq:tflex}
     (\p_j - \p_i)\cdot (\p'_j - \p'_i) 
    \ge 0, 
\end{eqnarray}
on the rest of the edges.

Any infinitesimal flex, in the 
sense of Definition \ref{def: inf flex},
is also an infinitesimal tensegrity flex.
We say that the packing is 
\defn{infinitesimally collectively jammed}
if the only infinitesimal tensegrity flexes $\p'$ are
are trivial infinitesimal flexes of $(G,\p)$.
There is always a $3$-dimensional space of 
trivial infinitesimal flexes.
\end{definition}
\begin{remark}
The terminology ``tensegrity flex'' comes from the
fact that a vector $\p'$ satisfying 
Equations \eqref{eq:tflex0}--\eqref{eq:tflex}
is an infinitesimal flex of a tensegrity 
structure  where the edges of
outer triangle are fixed-length bars, and all
of the internal edges are ``struts'' that can can increase, 
but not decrease, their
length during a flex.  (See the notes \cite{W16} for a 
detailed treatment of 
tensegrities.)
\end{remark}
These two notions are related by the following 
theorem of Connelly \cite{C08}.
\begin{theorem}\label{thm: dense vs inf jammed}
If a packing in a tri-cusp is infinitesimally
collectively jammed then it is 
locally maximally dense. If a packing in a tri-cusp is locally maximally
dense, then there is a sub-packing that is infinitesimally
collectively jammed.
\end{theorem}

When the tri-cusp 
packing is locally maximally dense, then the 
maximal
infinitesimally collectively jammed sub-packing forms a ``spine'' 
in which none of the disks can expand, even non-uniformly.
Disks that are not part of the spine are called 
``rattlers'' in the literature.  Rattlers 
can expand in any desired way.

\subsection{Bonus result: finite isostatic theorem}
The main result of this section is the following
``isostatic'' theorem:
\begin{theorem}
\label{cor:iso}
Let $P:=(\p,\rr)$ be 
a packing in a tri-cusp.
Suppose that the vector
$\bar{\rr}:=\{r_1, r_2, r_3, 
\bar{r}_4,\ldots,\bar{r}_{n-1}\}$
is generic in $\RR^{n-1}$.
Then $P$ 
cannot have more than 
$2n-2$ contacts. 
If $P$
is infinitesimally collectively jammed then
it has exactly $2n-2$ contacts.
\end{theorem}
The genericity assumption in Theorem \ref{cor:iso} is only on 
$\bar{\rr}$ (as opposed to $\rr$ in Theorem \ref{thm:main}), so 
the relative scale between the 
inner disks and the outer three need not be generic. 
The weaker genericity hypothesis 
will allow for the possibility of one 
(and only one) extra contact in the packing. The upper bound of $2n-2$ 
contacts does not depend on any notions of density or jamming.  When 
there are fewer than $2n-2$ contacts, the theorem says  
then there must exist a a non-trivial infinitesimal 
tensegrity flex, thus precluding infinitesimal collective jamming.

Combining Theorems \ref{cor:iso} and \ref{thm: dense vs inf jammed},
we obtain:
\begin{corollary}
Let $P := (\p,\rr)$ be a locally maximally dense packing 
in a tri-cusp
with
the vector $\bar{\rr}:=\{r_1, r_2, r_3, 
\bar{r}_4,...,\bar{r}_{n-1}\}$
generic in $\RR^{n-1}$.  Then $P$ has a sub-packing 
in a tri-cusp $P'$ of $n'$
disks  with $2n' - 2$ contacts among them
that is infinitesimally collectively jammed.
\end{corollary}

\begin{figure}[htbp]
    \centering
    \subfigure[]{\includegraphics[width=0.40\textwidth]{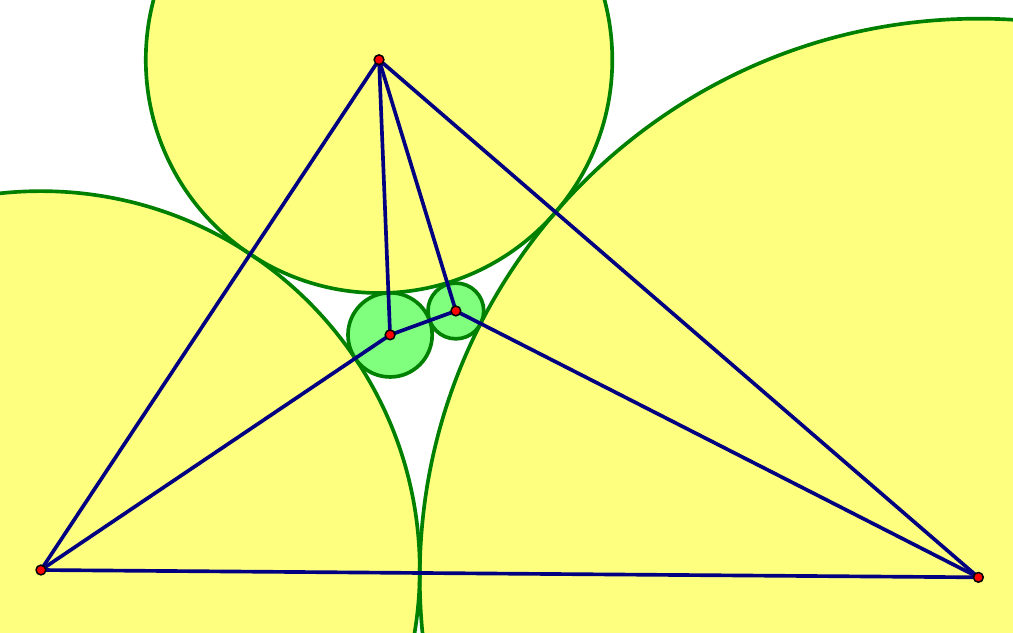}}
    \qquad\qquad
    \subfigure[]{\includegraphics[width=0.40\textwidth]{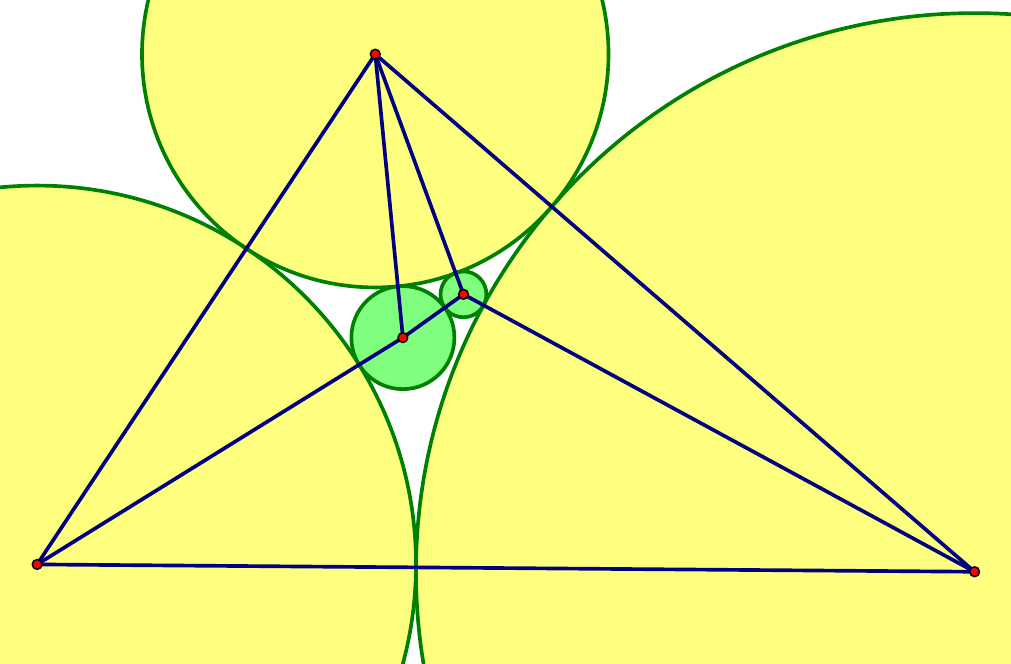}}
    \caption{
    (a) An infinitesimally collectively jammed packing in a tri-cusp
    with generic $\bar{\rr}$. The outer triangle consists of the three 
    large yellow disks, and there are $2$ green interior disks, so
    $n=5$.
    This packing has $8 = 2n-2$ contacts.
    (b)
    This packing in a tri-cusp also has generic $\bar{\rr}$ and $2n-2$ 
    contacts but it is not infinitesimally collectively jammed or even 
    maximally dense. The lower green disk can move slightly to the right and down, immediately breaking one contact,
    and then allow both green disks to expand.}
    \label{fig:nmd}
\end{figure}

\begin{proof}[Proof of Theorem \ref{cor:iso}]
Let us redefine $S_G$, in this section only,
to be the set of tri-cusp packings
with a fixed contact graph $G$.
As above (Proposition~\ref{prop:man}), $S_G$ (if not empty)
is smooth and of dimension
$3n-m$.

Let us redefine our projection $\pi$, in this section only, 
to be the map from $\RR^{3n}$
to $\RR^{n-1}$, that maps $(\p,\rr)$ to 
$\{r_1, r_2, r_3, 
\bar{r}_4,\ldots,\bar{r}_{n-1}\}$.  
With the new definition of $\pi$, the 
preimage  
$\pi^{-1}(\bar{\rr})$ in 
$S_G$ corresponds to packings with contact graph $G$,
where the outer three
disks are fixed (up to a Euclidean isometry)
and the internal radii ratios are fixed.

Next, 
we redefine a $\pi$-kernel vector at $(\p,\rr)\in S_G$
to be a tangent vector of $S_G$ that is in the kernel of the linearization of $\pi$.
Following the proof of  Lemma~\ref{lem:sard} above, except with an 
image of dimension $n-1$ instead of $n$, we see that 
the space of $\pi$-kernel vectors above a generic (and so regular) 
value in $\RR^{n-1}$ will be of dimension $2n-m+1$.  Since the 
space of $\pi$-kernel vectors above any point contains a $3$-dimensional 
subspace of trivial motions, we obtain $2n-m+1 \ge 3$.  Hence $m\le 2n - 2$,
giving us the first statement.

Let us now explore the implications of 
$m<2n-2$, which means that the
$\pi$-kernel is of dimension at least $ 4$.
We note for later that 
$\partial \bar{r}_j/\partial r_j = 1/r_n$
and
$\partial \bar{r}_j/\partial r_n = -r_j/r_n^2$.  Also, 
recall, from Equation \eqref{eq:stan}, that $(\p',\rr')$ 
is tangent to $S_G$ iff it satisfies 
\[
     (\p_j - \p_i)\cdot (\p'_j - \p'_i) =
    (r_j + r_i) (r'_j + r'_i).
\]
For a tangent vector $(\p',\rr')$ to be a $\pi$-kernel vector, it
must additionally satisify, for $j = 1,2,3$:
\[
    r'_j = 0 
\]
and, for $j = 4, \ldots, n$:
\begin{eqnarray}
\label{eq: piker scale-1}
 r'_j/r_n & = &  r'_n(r_j/r^2_n) \\
 \label{eq: piker scale-2}
 r'_j r_n & = & r'_n r_j 
\end{eqnarray}

Equations \eqref{eq: piker scale-1} and \eqref{eq: piker scale-2} imply that 
the $r'_j$ are either all non-positive or all non-negative.  After negating, 
if necessary, we conclude that, if $(\p',\rr')$ is a $\pi$-kernel vector,
then $\p'$ is an infinitesimal tensegrity flex.
(Note that the
converse is not true. An infinitesimal tensegrity flex vector $\p'$ does not
necessarily have a
corresponding  $\pi$-kernel vector. 
This ties into the discussion
of Section \ref{sec: no converse}, below.)

Thus if the $\pi$-kernel at $(\p,\rr)$ is of dimension
at least $4$,
then 
we will be able to find a non-trivial
infinitesimal tensegrity flex.  Hence, $P$ is not
infinitesimally collectively jammed.
\end{proof}
\begin{remark}
The lower bound aspect of Theorem~\ref{cor:iso},
namely that infinitesimal collective
jamming requires at least $2n-2$ 
contacts, is already established in~\cite{con88}. 
The results in \cite{con88} do not require genericity and
generalize to higher dimensions.
\end{remark}

\subsection{No converse}\label{sec: no converse}
There is a major difference between Theorems \ref{thm:main} and \ref{cor:iso} in 
that simply having generic radii and $2n-2$ contacts does not guarantee that a packing in a 
tri-cusp is infinitesimally collectively jammed or even locally maximally dense.
Figure \ref{fig:nmd}(b) shows an example.

Such examples show an essential difference between the inequalities 
in the jamming setup and the equalities defining the sticky behavior of 
sticky disks. In fact, if the disks in Figure \ref{fig:nmd}(b)
are forced to be sticky (but are still allowed to grow uniformly), then 
the example becomes rigid and infinitesimally rigid in 
an appropriate sense.

Interestingly, the papers \cite{PhysRevLett.114.135501,LESTER2018225,PhysRevLett.116.028301}
that partially motivated our interest in these problems 
use counting to analyze simulated jammed 
packings, essentially treating them as if the disks 
were sticky. 

\subsection{Relationship to the flat-torus isostatic theorem}
The isostatic theorem with a flat torus as the container
was proven in \cite{iso} using very general 
results of Guo \cite{G11} on circle patterns in 
piecewise-flat surfaces.  Such methods could also 
be applied to the tri-cusp setting.

It would be interesting to know if our methods can be extended 
to the torus.
Such an extension would require us to show that the corresponding $S_G$-set of 
$n$ disks on a a flat torus (with flexible metric and 
fixed affine structure) is smooth and of the expected 
codimension, $m$.  We don't know how to do that, and 
establishing the analogue of Lemma \ref{lem:noco}
for frameworks embedded in a torus
might be interesting in its own right.

\section{Open Problems}

\subsection{Existence}
Our paper proves certain properties about disk packings with generic 
radii. 
There certainly are classes of planar Laman graphs for which
we can build packings with generic radii. For a simple example, 
starting with a triangle, we can sequentially add on disks on
the exterior of the packing, each time using generic radii and
adding two contacts.
But importantly, we do not 
know about the existence of such packings 
for \emph{all} planar Laman graphs.
As in Remark~\ref{rem:kat},
we can see
that if $G$ is any planar graph,
then there must exist some disk packings with contact graph $G$.
But this reasoning does not
tell us about the genericity of the resulting radii.

\begin{question}\label{qu: ext}
Is the following claim true?
Let $G$ be planar 
and Laman. Then there is a packing $P$ with contact graph $G$
that has generic radii.
\end{question}
If there is a packing with generic radii,
then there will at least be an open ball of radii that can be used with $G$.

The generalization of Question \ref{qu: ext} to ball 
packings in three dimensions with contact graphs that 
are ``$3n - 6$ sparse'' 
(i.e., satisify the generalization of Maxwell's 
counting heuristic to dimension $3$)
appears to be false. The double banana 
graph~\cite[Figure 2]{CSS-nucleation} can
appear as the packing graph of $8$ balls, but it seems 
that such a packing will need carefully selected radii.
Of course the double banana is not a generically isostatic graph. 
So in 3D, we can weaken the  existence question 
and only consider contact graphs that are 
generically isostatic.

Here is an even stronger claim:
\begin{question}\label{qu: nearby}
Is the following claim true?
Let $P$ be any disk packing where its contact graph $G$ is planar 
and Laman. Then there is a nearby packing $P'$ with generic radii and the
same contact graph.
\end{question}
We can give a partial answer.
\begin{prop}\label{prop: inf rigid nearby}
Assuming that $P$ is infinitesimally rigid,  the answer to 
Question \ref{qu: nearby} is ``yes''.
\end{prop}
\begin{proof}[Proof sketch]
If $(G,\p)$ is an infinitesimally rigid framework of a Laman graph, 
then the vector $\bl$ of $m$ edge lengths of $(G,\p)$ is a regular point of the 
map that measures edge lengths.  The constant rank theorem (see, e.g., 
\cite[Theorem 9.32]{rudin}) then implies that there is a neighborhood $N$ of 
$\bl$ in $\RR^m$ consisting of edge length measurements arising
from frameworks close to  $(G,\p)$.

Next let $L$ be the linear map from $\RR^n$ to $\RR^m$ defined by 
$L_{ij}(\rr) := (r_i+r_j)$, where $ij$ ranges over the edges of $G$.
The image of $L$ is a linear space.
(Restricted to packing with contact graph $G$, the map $L$ measures the edge lengths of 
pairs of disks in contact, but we want the more general setting.)

By assumption, $P = (\p,\rr)$ is infinitesimally rigid.  
Hence, 
$N$ and $\bl$ as in the first paragraph are defined for 
the underlying framework $(G,\p)$.
Since $L(\rr) = \bl$, the image of $L$ intersects the interior of $N$.  Call this intersection $N'$.

For a sufficiently small perturbation $\rr'$ of $\rr$, we have 
$L(\rr')$ in $N'$.  This means there is a $\p'$ close to $\p$ 
so that $(G,\p')$ has edge lengths $L(\rr')$.  By picking $\rr'$
close enough to $\rr$, $\p'$ can be made close enough to $\p$ 
to guarantee that $(\p',\rr')$ is a packing with contact graph $G$
(i.e., with no new contacts or disk overlaps).

Since any neighborhood of $\rr$ contains a generic $\rr'$, we
obtain a nearby packing with generic radii.
\end{proof}
\begin{remark}
If $P$ is not infinitesimally rigid, the proof of Proposition  
\ref{prop: inf rigid nearby} above fails 
at the first step.  Without infinitesimal rigidity, the neighborhood 
$N$ may not exist, in which case the image of $L$ could be tangent 
to the set of achievable lengths at $\bl$.
\end{remark}

\begin{figure}[htbp]
    \centering
    \includegraphics[width=0.3\textwidth]
    {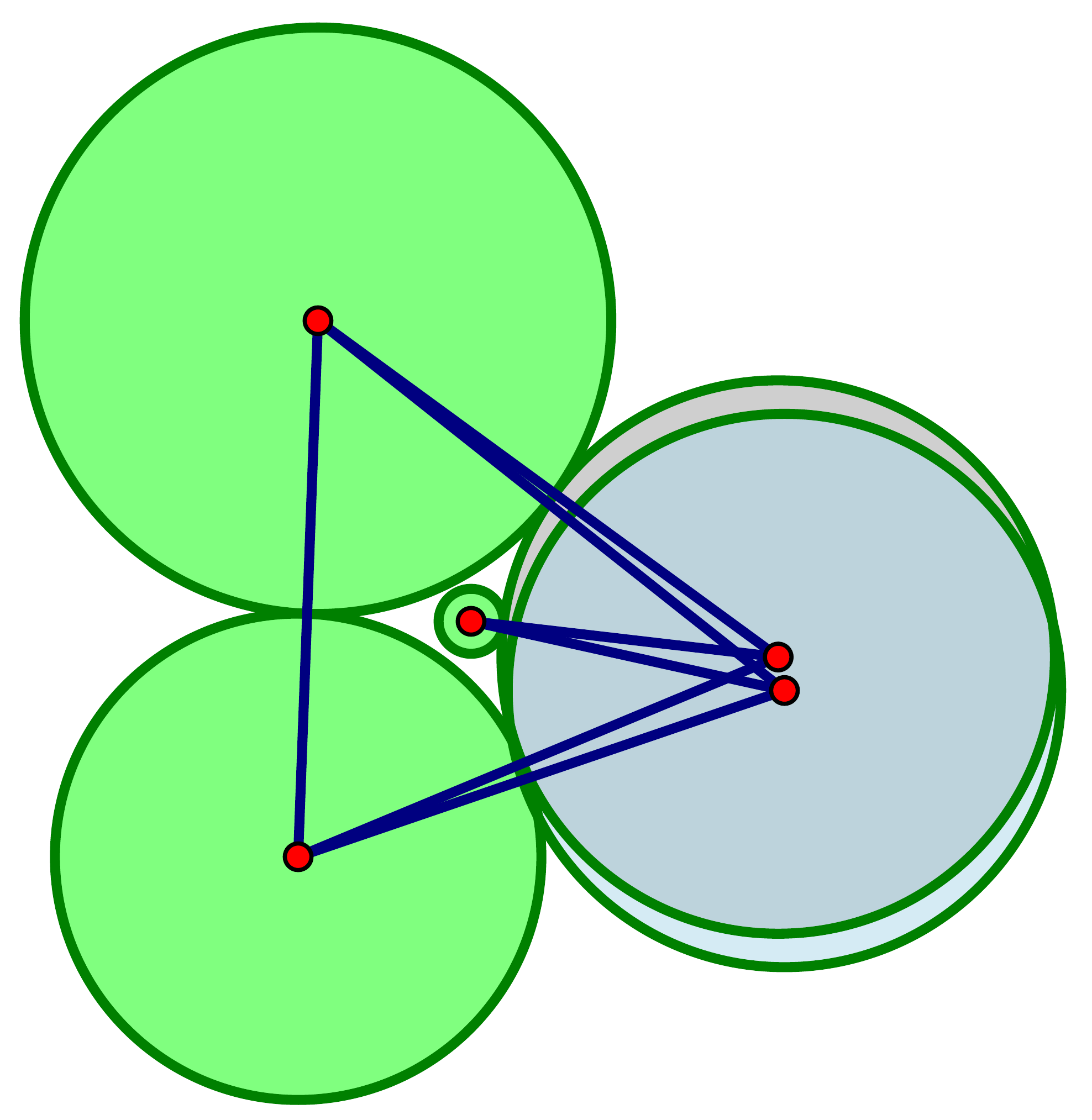}
    \caption{This disk configuration has two identical disks (separated slightly
    for visualization. In any nearby configuration with the same contact
    graph, the two gray disks must be identical. 
    Notably, this is not an example of a packing.}
    \label{fig:overlaps}
\end{figure}

The answer to Question \ref{qu: nearby} is ``no'', 
if we relax the packing-overlap constraint, 
as shown
in Figure~\ref{fig:overlaps}, even though all the disk contacts are external.

\subsection{Removing Inequalities}
The results of this paper rely on the fact that the underlying
framework is a planar embedding, which is guaranteed by the packing
inequalities.

In the \defn{algebraic setting}, we 
ignore the packing inequalities and enforce only 
the 
equality constraints $|\p_i-\p_j|^2=(r_i+r_j)^2$ on the edges of
our graph $G$. The analogous object to $S_G$ in the algebraic setting is an
algebraic variety (as opposed to a semi-algebraic set).

\begin{question}
In the algebraic setting, do the results of Theorem~\ref{thm:main} still hold?
\end{question}

\subsection{Three Dimensions}

The topics of this paper can be considered in three dimensions, where
disks are replaced with balls. The natural target contact number would
then become $3n-6$.

\begin{question}\label{qu: 3d}
Do the results of Theorem~\ref{thm:main} generalize to three dimensions?
\end{question}
The authors of \cite{meera-EASAL} conjecture that Question \ref{qu: 3d}
has a positive answer.

Interestingly, it conceivable that the Maxwell counting heuristic 
is sufficient for generic rigidity for generic radius ball packing. 
Maxwell counting is not sufficient for bar frameworks, with
that pesky double banana as a counter example. But it appears that the
double banana cannot appear as the contact graph of a ball packing with 
generic radii.

\section*{Acknowledgements}
We thank Meera Sitharam and Wai Yeung Lam for helpful feedback on 
an earlier version of this paper. The anonymous referees made 
a number of suggestions that improved the exposition.

\end{document}